\documentclass[12pt,reqno]{article}

\usepackage{amsmath,amsthm,amssymb,amsfonts}
\usepackage[T1]{fontenc}
\usepackage[utf8]{inputenc}
\usepackage[margin=1in]{geometry}
\usepackage{booktabs}
\usepackage{array}

\usepackage{tikz}
\usepackage{hyperref}

\hypersetup{
  colorlinks=true,
  linkcolor=blue!45!black,
  citecolor=blue!45!black,
  urlcolor=blue!55!black
}

\theoremstyle{plain}
\newtheorem{theorem}{Theorem}
\newtheorem{proposition}[theorem]{Proposition}
\newtheorem{lemma}[theorem]{Lemma}
\newtheorem{corollary}[theorem]{Corollary}

\theoremstyle{definition}

\newtheorem{question}[theorem]{Question}

\theoremstyle{remark}

\begin{document}
\newcommand{\supplementlink}{\href{run:sages_anonymous_slips_supplement.zip}{supplementary materials}}

\title{Sages and Their Anonymous Slips}
\author{Tanya Khovanova and Konstantin Knop}
\date{}
\maketitle

\begin{abstract}
Before the game, the sages agree on a finite set of symbols and on a rule for choosing one of them. They are then seated in an unknown order, either in a row or around a circle with a common clockwise direction. Each sage sees only the adjacent sage or sages and secretly writes one symbol. The slips are then mixed, so nobody learns who wrote which symbol; everyone learns only how many times each symbol occurred. The aim is that every sage should be able to independently reconstruct the entire seating. We ask for the smallest number of distinct symbols that makes this possible. We also consider a version in which the seats have fixed public labels, and each sage is told the label of their own seat.

For large numbers of sages, the minimum number of symbols is proportional to the square of the number of sages, up to a constant factor. Several small cases can be settled exactly. Four sages in a row need exactly two symbols, and so do five sages on a circle. Six sages on a circle need exactly three symbols, both with and without labeled seats; we give a short explicit rule. Five sages in a row with labeled seats also need exactly three symbols. Without seat labels, two symbols are impossible, and four are sufficient; whether three symbols suffice remains open. Whenever a small-case result relies on a computer check, the data and an independent exact verification program are included in the supplementary material.
\end{abstract}

\section{The puzzle}

We start with trivial cases. If three or four sages sit in a circle, they can reconstruct the seating order without using the announcement at all; equivalently, they may all write the same dummy symbol. The same is true for three sages in a row.
 
It is useful to begin with one concrete case. This is the smallest non-trivial number of sages in a circle.

\begin{quote}
\textbf{Puzzle.} Five sages, all of whom know one another, sit around a round table. Everyone agrees which direction is clockwise. Each sage sees only the immediate neighbor on the left and the immediate neighbor on the right. After looking, every sage secretly writes either $0$ or $1$ on a slip of paper.

The slips are collected and mixed. Nobody is told who wrote which slip; only the total number of $1$'s is announced. The sages may agree on a rule before they are seated. Can they choose such a rule so that, after the announcement, every sage can reconstruct the complete circular order?
\end{quote}

The answer is yes. Section~\ref{sec:five-circle} gives a short solution using a pentagon and its diagonals.

The same question behaves differently in a row. On the one hand, the sages at the ends know that they are sitting at the ends. On the other hand, they see only one neighbor rather than two. Which consideration is more useful? To give you the answer, five sages in a row cannot solve the problem with only two symbols. This contrast is a good reason to study both cases.

The problem grew out of a puzzle posed by Nikolai Chernyatiev discussed in~\cite{TKblog}. It is related to the familiar family of sages puzzles and hat-guessing games: each player sees only part of the situation, while the players may coordinate in advance. Winkler's collection~\cite{Winkler} and the paper of Butler, Hajiaghayi, Kleinberg, and Leighton~\cite{BHKhL} provide useful background.

There is one important difference. In a usual hat problem, the players try to infer information that was chosen independently of them. Here the sages create the public information themselves: each writes one symbol. The public information is only the number of occurrences of each symbol. The question is how much of the unknown seating order can still be recovered from these anonymous counts. Chernyatiev's original version retained the author of each slip; in this paper we focus on the anonymous version.

The paper is organized as follows. Section~\ref{sec:notation} introduces the variants and notation. Sections~\ref{sec:counting} and~\ref{sec:general} give general lower and upper bounds. The remaining sections settle the small cases in Table~\ref{tab:small-results}, except for $5R3E$, which remains open.

\section{Rules and notation}
\label{sec:notation}

We use compact names for the variants we study. A code such as $6C3L$ has four parts: the number of sages, the shape of the seating, the number of allowed symbols, and the information available about the seat. Table~\ref{tab:code} lists all notation used in codes.

\begin{table}[ht]
\centering
\begin{tabular}{c|c|p{0.72\textwidth}}
\toprule
position & notation & meaning\\
\midrule
1 & $S$ & the number of sages\\
2 & $R$ & the sages sit in a row\\
2 & $C$ & the sages sit in a circle, and everyone distinguishes clockwise from counterclockwise; in the $E$-version, orders differing only by rotation are identified, while in the $L$-version the fixed seat labels distinguish them\\
3 & $n$ & the number of symbols the sages are allowed to write\\
4 & $E$ & the basic version: the slips are anonymous, and a sage is told nothing about which physical seat they occupy\\
4 & $L$ & the labeled-seat version: every seat has a distinct fixed public label, and each sage is told the label of their own seat; the slips are still anonymous\\
\bottomrule
\end{tabular}
\caption{Notation for the code.}
\label{tab:code}
\end{table}

For example, $5R3E$ means five sages in a row, with three available symbols and no seat labels. The code $6C2L$ means six sages on a circle, two available symbols, and labeled seats.

For each geometry and information model, we study the minimum number of symbols needed for a given number $S$ of sages. The four functions are defined in Table~\ref{tab:notation}.
\begin{table}[ht]
\centering
\begin{tabular}{c|p{0.72\textwidth}}
\toprule
notation & meaning\\
\midrule
$r_E(S),\ r_L(S)$ & the minimum possible values of $n$ for $S$ sages in a row, in the $E$- and $L$-versions\\
$c_E(S),\ c_L(S)$ & the corresponding minima for $S$ sages on a circle\\
\bottomrule
\end{tabular}
\caption{Notation for the minima.}
\label{tab:notation}
\end{table}

Every sage knows all the names and distinguishes left from right. At an end of a row, a sage sees one neighbor and knows which end they occupy; in the interior, a sage sees the ordered pair of neighbors. We define a sage's \emph{local view} to consist of their own name together with this neighbor information. In the $L$-version, the local view also includes the label of the sage's own seat. The agreed rule may use the entire local view, but it contains no random choices. We call this agreed rule a \emph{protocol}.

After the slips are mixed, everyone learns only how many times each symbol occurs. Equivalently, they learn the \emph{multiset} of the $S$ symbols: a bag in which repetitions are remembered but authors are not. A protocol succeeds if every sage can combine this public multiset with their local information and recover the whole seating. Any protocol for an $E$-problem also works for the corresponding $L$-problem, because a sage may simply ignore the seat label.

Table~\ref{tab:small-results} summarizes the small cases proved in this paper. The only unresolved value is $r_E(5)$

\begin{table}[ht]
\centering
\begin{tabular}{c|c|c|l}
\toprule
seating & no seat labels ($E$) & labeled seats ($L$) & proved in\\
\midrule
$4$ sages in a row & $2$ & $2$ & Section~\ref{sec:four-row} \\
$5$ sages in a row & $3$ or $4$ & $3$ & Section~\ref{sec:five-row-more} \\
$5$ sages on a circle & $2$ & $2$ & Section~\ref{sec:five-circle} \\
$6$ sages on a circle & $3$ & $3$ & Sections~\ref{sec:six-circle} and~\ref{sec:six-circle-three}\\
\bottomrule
\end{tabular}
\caption{Minimum numbers of symbols in the small cases discussed below.}
\label{tab:small-results}
\end{table}

\section{A first counting bound}
\label{sec:counting}

Fix a sage and their local view. Call a complete seating \emph{compatible} if it gives that sage this local view.

\begin{lemma}
\label{lem:compatible}
The number of compatible seatings is as follows.
\begin{enumerate}
\item In the $E$ row model it is $(S-2)!$.
\item In the $L$ row model it is $(S-2)!$ for an end sage and $(S-3)!$ for an interior sage.
\item In either circular model it is $(S-3)!$.
\end{enumerate}
\end{lemma}

\begin{proof}
An end sage in a row sees one neighbor, after which the remaining $S-2$ sages can be placed arbitrarily. In the $E$ row model, an interior sage sees an ordered triple consisting of the left neighbor, the sage, and the right neighbor. This triple can occupy any of $S-2$ consecutive positions, and the other $S-3$ sages can be permuted freely, giving $(S-2)(S-3)!=(S-2)!$ possibilities. In the $L$ row model, the seat label fixes the position of the triple.

In a circle, fixing the sage and the two ordered neighbors leaves the other $S-3$ sages to be arranged along the remaining arc in an arbitrary order.
\end{proof}

The simplest lower bound comes from counting how many different public announcements are available. There is one useful observation before we count: a sage already knows the symbol on their own slip. In their mind, they may therefore remove one known copy of that symbol from the public multiset. What remains is a multiset of only $S-1$ symbols. For two-symbol problems, this observation helps with the answer in the first interesting cases.

\begin{proposition}[The counting bound]
\label{prop:counting}
For $S$ sages in a row using $n$ symbols, in either the $E$- or $L$-version, solvability requires
\[
  \binom{n+S-2}{S-1}\ge (S-2)!.
\]
For $S$ sages on a circle using $n$ symbols, in either version, solvability requires
\[
  \binom{n+S-2}{S-1}\ge (S-3)!.
\]
\end{proposition}

\begin{proof}
Fix a sage and fix what that sage sees. Their own output is then known. After removing it, the public announcement is a bag of $S-1$ symbols of $n$ possible types. The number of such bags is
\[
  \binom{n+S-2}{S-1},
\]
the usual count for choosing $S-1$ objects of $n$ types when repetition is allowed. If two compatible seatings produced the same bag, this sage could not tell them apart. Thus there must be at least as many bags as compatible seatings. For a row, looking at an end sage gives the first inequality; for a circle, Lemma~\ref{lem:compatible} gives the second.
\end{proof}

This argument can only rule problems out; it does not construct a strategy. Even so, it gives a useful first map. For $n=2,3,4$, the largest values of $S$ not ruled out by counting are arranged in Table~\ref{tab:counting}.

\begin{table}[ht]
\centering
\begin{tabular}{c|ccc}
\toprule
number of symbols $n$ & $2$ & $3$ & $4$\\
\midrule
largest $S$ not ruled out for a row & $4$ & $5$ & $6$\\
largest $S$ not ruled out for a row a circle & $6$ & $7$ & $8$\\
\bottomrule
\end{tabular}
\caption{The elementary counting boundary for anonymous slips.}
\label{tab:counting}
\end{table}

With two symbols, removing one's own slip leaves only $S$ possible counts: $0,1,\ldots,S-1$ copies of one of the symbols. Therefore, five sages in a row are already impossible by counting. For six sages on a circle, the two sides of the inequality are both $6$:
\[
   \binom{6}{5}= 6=(6-3)!.
\]
So the simple count neither proves nor disproves that case; we return to it in Section~\ref{sec:six-circle}.

\section{A general construction}
\label{sec:general}

There is a simple construction that works for every number of sages. Give a separate symbol to every unordered pair of names. The sages will use these symbols to record a pair of themselves and one of their neighbors. Although the slips are anonymous, the collection of pair-symbols still reveals the whole path or cycle formed by the seating.

\begin{theorem}
\label{thm:pairs}
For every $S\ge3$, both the row and circle problems can be solved in the $E$-version using
\[
  \binom{S}{2}
\]
symbols. The same protocols also solve the corresponding $L$-problems.
\end{theorem}

\begin{proof}
Assign a different symbol $c_{\{x,y\}}$ to every unordered pair of sages $\{x,y\}$.

First put the sages in a circle. Every sage writes the symbol belonging to the pair consisting of themselves and their right neighbor. If the circular order is
\[
  b_1,b_2,\ldots,b_S,
\]
then the announced multiset is
\[
 \bigl\{c_{\{b_1,b_2\}},c_{\{b_2,b_3\}},\ldots,
 c_{\{b_S,b_1\}}\bigr\}.
\]
Thus the slips reveal exactly which pairs of sages sit next to each other. Those pairs form a single cycle. There are two ways to travel around that cycle, and a sage's knowledge of which neighbor is on the left chooses the correct direction.

Now put the sages in a row. Every sage except the rightmost one writes the symbol belonging to the pair consisting of themselves and their right neighbor. The rightmost sage writes the symbol belonging to the pair consisting of themselves and their left neighbor. For the row
\[
  b_1,b_2,\ldots,b_S,
\]
the multiset is
\[
 \bigl\{c_{\{b_1,b_2\}},c_{\{b_2,b_3\}},\ldots,
 c_{\{b_{S-1},b_S\}},c_{\{b_{S-1},b_S\}}\bigr\}.
\]
After duplicates are ignored, these symbols tell us exactly which pairs are adjacent in the row, so they form one chain containing all the names. The one symbol that occurs twice marks the last pair in the row. Of the two names in that pair, exactly one occurs in no other pair; that sage is the rightmost one. The chain is therefore determined from left to right.
\end{proof}

We can now estimate the growth of our functions $r_E(S), r_L(S), c_E(S), c_L(S)$.

\begin{corollary}
\label{cor:theta}
As $S\to\infty$, all four minimum numbers of symbols grow quadratically:
\[
  r_E(S),\ r_L(S),\ c_E(S),\ c_L(S)=\Theta(S^2).
\]
Here $\Theta(S^2)$ means that each quantity is bounded above and below by positive constant multiples of $S^2$ for all sufficiently large $S$. More precisely,
\[
 \left(\frac1{e^2}+o(1)\right)S^2
 \le r_L(S)\le r_E(S)\le \binom{S}{2},
\]
and
\[
 \left(\frac1{e^2}+o(1)\right)S^2
 \le c_L(S)\le c_E(S)\le \binom{S}{2}.
\]
\end{corollary}

\begin{proof}
The upper bounds come from Theorem~\ref{thm:pairs}. For the lower bounds, apply Proposition~\ref{prop:counting} and Stirling's formula. The shift from $(S-2)!$ to $(S-3)!$ changes only lower-order terms. If $n=\Theta(S^2)$, then
\[
 \log\binom{n+S-2}{S-1}
 = (S-1)\log n-\log((S-1)!)+O\!\left(\frac{S^2}{n}\right).
\]
Combining this estimate with Stirling's formula shows that $n$ must be at least $(1/e^2+o(1))S^2$.
\end{proof}

The two bounds have the same order of growth: the number of required symbols is proportional to $S^2$ up to constant factors. The constant is not known. Counting forces roughly $S^2/e^2$ symbols, while the explicit pair-encoding construction uses roughly $S^2/2$. The comparison is in Table~\ref{tab:asymptotic}.

\begin{table}[ht]
\centering
\begin{tabular}{c|c|c}
\toprule
quantity & asymptotic lower bound & explicit upper bound\\
\midrule
$r_E(S)$, $r_L(S)$, $c_E(S)$, $c_L(S)$  & $(1/e^2+o(1))S^2$ & $\binom{S}{2}$\\
\bottomrule
\end{tabular}
\caption{The leading constant lies somewhere between $1/e^2$ and $1/2$.}
\label{tab:asymptotic}
\end{table}

\section{Four sages in a row}
\label{sec:four-row}

Let the four sages be $A,B,C,D$. Before the game, they agree on the following arrows between their names:
\[
  A\longrightarrow B\longrightarrow C\longrightarrow A,
  \qquad
  D\longrightarrow A,\quad D\longrightarrow B,\quad
  D\longrightarrow C.
\]
Thus $A,B,C$ form a rock--paper--scissors cycle, while every arrow from $D$ points outward. Define $q(x,y)=1$ if the agreed arrow goes from $x$ to $y$, and $q(x,y)=0$ otherwise.

Here is the protocol. The two sages at the ends always write $0$. An interior sage with left neighbor $L$ and right neighbor $R$ writes $q(L,R)$.

The following elementary property of the arrow diagram is exactly what the decoding needs.

\begin{lemma}
\label{lem:tournament}
For every ordering $x,y,u,v$ of the four sages,
\[
  q(x,u)+q(y,v)\ne q(x,v)+q(y,u).
\]
\end{lemma}

\begin{proof}
Consider the location of $D$. If $D=x$, then $q(x,u)=q(x,v)=1$, while among $A,B,C$ the sage $y$ defeats exactly one of $u,v$. The case $D=y$ is the same. If $D=u$, then $q(x,u)=q(y,u)=0$, while exactly one of $x,y$ defeats $v$ in the directed $3$-cycle. The case $D=v$ is symmetric.
\end{proof}

\begin{proposition}
\label{prop:4R2E}
The puzzle $4R2E$ is solvable.
\end{proposition}

\begin{proof}
Suppose first that an interior sage $x$ sees $L$ on the left and $R$ on the right, and let $u$ be the unseen sage. The two compatible rows are
\[
  L\,x\,R\,u
  \qquad\text{and}\qquad
  u\,L\,x\,R.
\]
The contribution of $x$ is $q(L,R)$ in both rows. The other interior contribution is $q(x,u)$ in the first row and $q(u,x)$ in the second. Exactly one of these is $1$, so the announced totals differ.

Now suppose the left-end sage $x$ sees $y$, leaving $u,v$ unseen. The two possible rows are
\[
  x\,y\,u\,v
  \qquad\text{and}\qquad
  x\,y\,v\,u.
\]
Their numbers of $1$'s are
\[
  q(x,u)+q(y,v)
  \qquad\text{and}\qquad
  q(x,v)+q(y,u),
\]
which are different by Lemma~\ref{lem:tournament}. The right-end case is symmetric. Hence every sage can select the unique compatible row with the announced total.
\end{proof}

\section{Five sages in a circle}
\label{sec:five-circle}

The five-sage circle has an especially pleasant solution: no table, no search, just a pentagon.

In advance, place the five sages' names at the vertices of a regular pentagon in an agreed order. Each pair of names is now joined either by a side or by a diagonal. We call this pentagon the \textit{name pentagon}.

Here is the protocol: a sage writes $1$ if the segment joining their own name to their left neighbor is a diagonal, and writes $0$ if it is a side, as illustrated in Figure~\ref{fig:name-pentagon}. Thus, the announced number of $1$'s is simply the number of seating adjacencies that are diagonals of the name pentagon.

\begin{figure}[ht]
\centering
\begin{tikzpicture}[scale=1.15,line cap=round,line join=round]
  \foreach \i/\name in {0/A,1/B,2/C,3/D,4/E}{
    \coordinate (P\i) at ({90-72*\i}:1.35);
    \fill (P\i) circle (1.5pt);
    \node at ({90-72*\i}:1.62) {$\name$};
  }
  \draw[gray!70,line width=.7pt]
    (P0)--(P1)--(P2)--(P3)--(P4)--cycle;
  \draw[line width=.9pt]
    (P0)--(P2)--(P4)--(P1)--(P3)--cycle;
\end{tikzpicture}
\caption{The five diagonals form another $5$-cycle. A sage writes $1$ when the seating edge to the left neighbor belongs to this diagonal cycle.}
\label{fig:name-pentagon}
\end{figure}
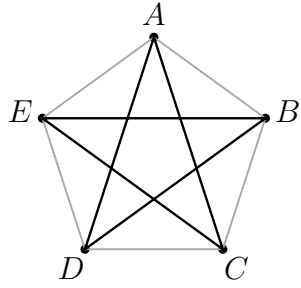

\begin{proposition}
\label{prop:5C2E}
The puzzle $5C2E$ is solvable.
\end{proposition}

\begin{proof}
Fix a sage $x$, with left neighbor $L$ and right neighbor $R$. Let $u$ and $v$ be the two unseen sages. Up to rotation, the two possible oriented circular orders are
\[
  x,R,u,v,L
  \qquad\text{and}\qquad
  x,R,v,u,L.
\]
These two seating cycles have the three common edges $xR$, $xL$, and $uv$. Their remaining edges are, respectively,
\[
  \{Ru,vL\}
  \qquad\text{and}\qquad
  \{Rv,uL\}.
\]
Each set simply pairs the four names $L,R,u,v$ into two disjoint pairs, and the two pairings are different.

The five diagonals of the name pentagon form another $5$-cycle. After removing $x$, those diagonals form a path through the other four names. There are exactly three ways to split four names into two pairs, and these three pairings contain, respectively, $0$, $1$, and $2$ edges of that path. Hence two different pairings use different numbers of pentagon diagonals. The two candidate seatings therefore give different numbers of $1$'s, so $x$ knows which seating is the real one.
\end{proof}

\section{Why two symbols fail for six sages on a circle}
\label{sec:six-circle}

For six sages on a circle, the counting bound is exactly tight: after removing their own slip, a sage sees one of six possible counts, and six complete circular orders are compatible with the same local information. Thus counting alone cannot decide the problem.

We prove something stronger than needed. Give every seat a fixed label and tell each sage the label of their own seat. Even with this extra information, two symbols are not enough.

\begin{theorem}
\label{thm:6C2L}
The puzzle $6C2L$ has no deterministic solution.
\end{theorem}

The proof is a double-counting argument and does not require computation.

\begin{proof}
Number the six seats clockwise and regard all $6!=720$ assignments of the six distinct sages to these labeled seats as distinct. A \emph{local view} specifies a sage, their seat, and the sages immediately to their left and right. There are $6\cdot6\cdot5\cdot4=720$ such views, and exactly $3!=6$ full seatings are compatible with each view.

Suppose a successful binary protocol exists, with symbols $0$ and $1$. Write $T(\pi)$ for the total number of $1$'s in a full seating $\pi$. For a local view $v$, let $c(v)\in\{0,1\}$ be the bit written by its observer. Because the six compatible seatings must give distinct public totals, and the observer's bit is fixed, their totals are exactly
\[
  c(v),\ c(v)+1,\ \ldots,\ c(v)+5.
\]
Let $V_0$ and $V_1$ be the numbers of local views with $c(v)=0$ and $c(v)=1$, respectively. Let $N_t$ be the number of complete labeled seatings with $T(\pi)=t$.

Count pairs $(v,\pi)$ in which $\pi$ is compatible with $v$ and $T(\pi)=1$. Every local view, regardless of its bit, occurs in exactly one such pair. On the other hand, a seating with total $1$ has five observers writing $0$ and one writing $1$. Counting separately by the observer's bit yields
\[
  V_0=5N_1,\qquad V_1=N_1.
\]
Now perform the analogous count for pairs with $T(\pi)=2$. Again, every local view occurs exactly once, since $2$ belongs to both possible sets of six totals. Each seating with total $2$ has four observers writing $0$ and two writing $1$. Thus
\[
  V_0=4N_2,\qquad V_1=2N_2.
\]
Since $V_0+V_1=720>0$, the first pair of identities implies $V_0=5V_1$ with $V_1>0$, while the second implies $V_0=2V_1$. This is impossible. Therefore, no successful protocol exists.
\end{proof}

We can now summarize all problems with two available symbols. The four-sage row and the five-sage circle are the last solvable cases; counting rules out all larger cases except the six-sage circle, which was settled by Theorem~\ref{thm:6C2L}.

\begin{theorem}[Two-symbol classification]
\label{thm:binary}
When only two symbols are available, the complete picture is
\begin{center}
\begin{tabular}{c|c|c}
\toprule
geometry & $E$-version & $L$-version\\
\midrule
row & solvable exactly for $S\le4$ & solvable exactly for $S\le4$\\
circle & solvable exactly for $S\le5$ & solvable exactly for $S\le5$\\
\bottomrule
\end{tabular}
\end{center}
\end{theorem}

\begin{proof}
Rows with at most three sages and circles with at most four sages need no public information. Propositions~\ref{prop:4R2E} and~\ref{prop:5C2E} give the two nontrivial positive boundary cases, and the same protocols work with labeled seats.

For two symbols,
\[
  \binom{2+S-2}{S-1}=S.
\]
The row inequality in Proposition~\ref{prop:counting} fails for every $S\ge5$. The circular inequality fails for every $S\ge7$. The remaining circular case $S=6$ is impossible even with labeled seats by Theorem~\ref{thm:6C2L}.
\end{proof}

\section{Six sages in a circle with three symbols}
\label{sec:six-circle-three}

Two symbols do not suffice for six sages on a circle. With three symbols, however, there is a short explicit rule. The rule below may look slightly artificial at first, but it uses only the sage's name and the two visible neighbors; no table of $120$ cases is needed.

\begin{proposition}
\label{prop:6C3E}
Both $6C3E$ and $6C3L$ are solvable. Consequently,
\[
  c_E(6)=c_L(6)=3.
\]
\end{proposition}

\begin{proof}
Number the sages $0,1,\ldots,5$, with arithmetic modulo $6$. Suppose sage $x$ sees $L$ on the left and $R$ on the right. Choose $a,b\in\{1,2,3,4,5\}$ so that
\[
  a\equiv L-x\pmod 6,
  \qquad
  b\equiv R-x\pmod 6.
\]
Thus $a$ and $b$ record where the two neighbors lie relative to $x$ in the agreed cyclic numbering of the names.

The rule is as follows.
\begin{itemize}
\item If $b=3$, write $3$.
\item Otherwise compare $2a+3b$ with $15$. For an even $x$, write $1$ when $2a+3b<15$ and $2$ when $2a+3b>15$.
\item For an odd $x$, interchange the roles of $1$ and $2$.
\end{itemize}
For an allowed pair $a\ne b$ with $b\ne3$, equality $2a+3b=15$ never occurs, so the rule is unambiguous.

The rule has two simple symmetries. If every name is increased by $1$ modulo $6$, the relative numbers $a,b$ stay the same, and the roles of symbols $1$ and $2$ are exchanged. If every name is replaced by its negative modulo $6$, the inequality is reversed, and again $1$ and $2$ are exchanged, while symbol $3$ stays fixed. Rotations and reflections of the six names are generated by these two operations, so the $120$ possible local views fall into only $10$ symmetry classes.

For any fixed local view, the three unseen sages can occur in $3!=6$ orders. To prove that the rule works, one only has to check that these six orders always give six different public multisets. The exact check is included in the \supplementlink; it succeeds for all local views. Thus every sage can identify the unique compatible circular order. The same rule solves $6C3L$, because the seat label can simply be ignored. Together with Theorem~\ref{thm:6C2L}, this proves $c_E(6)=c_L(6)=3$.
\end{proof}

\section{Five sages in a row: three symbols or four?}
\label{sec:five-row-more}

The only unresolved small case in Table~\ref{tab:small-results} is five sages in a row without seat labels. We know
\[
  3\le r_E(5)\le4,
\]
and we do not know whether the answer is $3$ or $4$.

\begin{proposition}
\label{prop:computer-protocols}
The following two row problems are solvable.
\begin{enumerate}
\item $5R4E$;
\item $5R3L$.
\end{enumerate}
\end{proposition}

\begin{proof}
Both strategies are specified by explicit lookup tables in Appendices~\ref{app:5R4E} and~\ref{app:5R3L}. For each of the $5!=120$ possible rows and each sage, verification requires that no two rows compatible with that sage's local observation produce the same public multiset. The $5R3L$ protocol may additionally depend on the sage's publicly labeled seat.

The exact exhaustive checker included in the \supplementlink verifies this condition for every sage and every local view in both protocols. Hence both protocols succeed.
\end{proof}

Together with Theorem~\ref{thm:binary}, the second protocol shows that
\[
  r_L(5)=3.
\]
What we do not know is whether those seat labels are really needed.

\begin{question}
\label{q:5R3E}
Can five sages in a row solve the anonymous-slip puzzle with three symbols without seat labels? In other words, is $5R3E$ solvable?
\end{question}

\section{Further questions}

Even these tiny examples leave several natural problems.

\begin{question}
Can the gap in Corollary~\ref{cor:theta} be narrowed? The present lower bound has leading constant $1/e^2$, while the explicit construction has leading constant $1/2$.
\end{question}

\begin{question}
For a fixed number of symbols, does solvability for $S$ sages imply solvability for fewer sages? The anonymous multiset makes the usual ``remove one player'' reductions unexpectedly awkward.
\end{question}

\section{AI Usage}

We used ChatGPT Pro for exploratory computation, verification code, and editorial assistance. The authors checked the resulting mathematical arguments and computational conclusions.

\appendix

\section{A four-symbol protocol for \texorpdfstring{$5R4E$}{5R4E}}
\label{app:5R4E}

Number the sages $0,1,2,3,4$. If sage $x$ sees left neighbor $L$ and
right neighbor $R$, they write $f_x(L,R)$. The symbol ``$-$'' means that
the corresponding side of the row is empty, and a dot denotes an
impossible local view. The five tables are as follows.

{\small
\setlength{\arraycolsep}{2.2pt}
\[
\begin{array}{c@{\hspace{4em}}c@{\hspace{4em}}c}
\begin{array}{c}
x=0\\[-1pt]
\begin{array}{c|cccccc}
L\backslash R & -&0&1&2&3&4\\\hline
- & \cdot & \cdot & 1 & 1 & 2 & 2\\
0 & \cdot & \cdot & \cdot & \cdot & \cdot & \cdot\\
1 & 4 & \cdot & \cdot & 1 & 4 & 2\\
2 & 3 & \cdot & 3 & \cdot & 4 & 3\\
3 & 1 & \cdot & 1 & 4 & \cdot & 3\\
4 & 2 & \cdot & 3 & 3 & 4 & \cdot
\end{array}
\end{array} &
\begin{array}{c}
x=1\\[-1pt]
\begin{array}{c|cccccc}
L\backslash R & -&0&1&2&3&4\\\hline
- & \cdot & 3 & \cdot & 2 & 1 & 2\\
0 & 1 & \cdot & \cdot & 4 & 3 & 3\\
1 & \cdot & \cdot & \cdot & \cdot & \cdot & \cdot\\
2 & 1 & 4 & \cdot & \cdot & 2 & 2\\
3 & 2 & 3 & \cdot & 2 & \cdot & 2\\
4 & 3 & 1 & \cdot & 3 & 4 & \cdot
\end{array}
\end{array} &
\begin{array}{c}
x=2\\[-1pt]
\begin{array}{c|cccccc}
L\backslash R & -&0&1&2&3&4\\\hline
- & \cdot & 3 & 2 & \cdot & 4 & 1\\
0 & 3 & \cdot & 4 & \cdot & 1 & 3\\
1 & 2 & 1 & \cdot & \cdot & 4 & 4\\
2 & \cdot & \cdot & \cdot & \cdot & \cdot & \cdot\\
3 & 3 & 4 & 1 & \cdot & \cdot & 3\\
4 & 3 & 4 & 1 & \cdot & 1 & \cdot
\end{array}
\end{array}\\
\noalign{\vskip 0.6em}
\multicolumn{3}{c}{%
\begin{array}{c@{\hspace{4em}}c}
\begin{array}{c}
x=3\\[-1pt]
\begin{array}{c|cccccc}
L\backslash R & -&0&1&2&3&4\\\hline
- & \cdot & 3 & 3 & 2 & \cdot & 1\\
0 & 4 & \cdot & 1 & 3 & \cdot & 1\\
1 & 3 & 2 & \cdot & 4 & \cdot & 2\\
2 & 1 & 1 & 2 & \cdot & \cdot & 4\\
3 & \cdot & \cdot & \cdot & \cdot & \cdot & \cdot\\
4 & 2 & 2 & 1 & 4 & \cdot & \cdot
\end{array}
\end{array} &
\begin{array}{c}
x=4\\[-1pt]
\begin{array}{c|cccccc}
L\backslash R & -&0&1&2&3&4\\\hline
- & \cdot & 4 & 3 & 3 & 1 & \cdot\\
0 & 3 & \cdot & 4 & 3 & 1 & \cdot\\
1 & 1 & 2 & \cdot & 3 & 1 & \cdot\\
2 & 3 & 2 & 4 & \cdot & 2 & \cdot\\
3 & 4 & 2 & 2 & 1 & \cdot & \cdot\\
4 & \cdot & \cdot & \cdot & \cdot & \cdot & \cdot
\end{array}
\end{array}
\end{array}}
\end{array}
\]
}

\section{A three-symbol protocol for \texorpdfstring{$5R3L$}{5R3L}}
\label{app:5R3L}

Again number the sages $0,1,2,3,4$, and number the seats $1,2,3,4,5$
from left to right. For a sage $x$ in position $i$, write
\[
  F_i(x;L,R)\in\{1,2,3\},
\]
where $L$ or $R$ is absent at an end. A dot denotes an impossible local
view. The following tables give a successful protocol. The two end
positions are listed first, followed by the three interior positions.

\paragraph{End positions.}

{\small
\setlength{\arraycolsep}{3.5pt}
\[
\begin{array}{c@{\hspace{4em}}c}
\text{Position 1} & \text{Position 5}\\[2pt]
\begin{array}{c}
F_1(x,R)\\[-1pt]
\begin{array}{c|ccccc}
x\backslash R & 0&1&2&3&4\\\hline
0 & \cdot & 1 & 1 & 1 & 3\\
1 & 2 & \cdot & 1 & 1 & 3\\
2 & 3 & 3 & \cdot & 2 & 2\\
3 & 3 & 3 & 1 & \cdot & 2\\
4 & 1 & 1 & 1 & 1 & \cdot
\end{array}
\end{array} &
\begin{array}{c}
F_5(x,L)\\[-1pt]
\begin{array}{c|ccccc}
x\backslash L & 0&1&2&3&4\\\hline
0 & \cdot & 1 & 2 & 3 & 1\\
1 & 1 & \cdot & 3 & 2 & 1\\
2 & 3 & 2 & \cdot & 1 & 1\\
3 & 2 & 3 & 1 & \cdot & 2\\
4 & 1 & 3 & 2 & 1 & \cdot
\end{array}
\end{array}
\end{array}
\]
}

\paragraph{Position 2.}

{\small
\setlength{\arraycolsep}{2.5pt}
\setlength{\abovedisplayskip}{\abovedisplayshortskip}
\setlength{\belowdisplayskip}{4pt}
\setlength{\belowdisplayshortskip}{4pt}
\[
\begin{array}{c@{\hspace{4em}}c@{\hspace{4em}}c}
\begin{array}{c}
F_{2}(0;L,R)\\[-1pt]
\begin{array}{c|ccccc}
L\backslash R & 0&1&2&3&4\\\hline
0 & \cdot & \cdot & \cdot & \cdot & \cdot\\
1 & \cdot & \cdot & 1 & 3 & 3\\
2 & \cdot & 1 & \cdot & 3 & 1\\
3 & \cdot & 3 & 3 & \cdot & 2\\
4 & \cdot & 3 & 2 & 1 & \cdot
\end{array}
\end{array} &
\begin{array}{c}
F_{2}(1;L,R)\\[-1pt]
\begin{array}{c|ccccc}
L\backslash R & 0&1&2&3&4\\\hline
0 & \cdot & \cdot & 2 & 1 & 2\\
1 & \cdot & \cdot & \cdot & \cdot & \cdot\\
2 & 2 & \cdot & \cdot & 1 & 3\\
3 & 2 & \cdot & 1 & \cdot & 2\\
4 & 1 & \cdot & 1 & 2 & \cdot
\end{array}
\end{array} &
\begin{array}{c}
F_{2}(2;L,R)\\[-1pt]
\begin{array}{c|ccccc}
L\backslash R & 0&1&2&3&4\\\hline
0 & \cdot & 3 & \cdot & 1 & 2\\
1 & 1 & \cdot & \cdot & 2 & 2\\
2 & \cdot & \cdot & \cdot & \cdot & \cdot\\
3 & 3 & 2 & \cdot & \cdot & 1\\
4 & 1 & 1 & \cdot & 3 & \cdot
\end{array}
\end{array}\\
\noalign{\vskip 0.6em}
\multicolumn{3}{c}{%
\begin{array}{c@{\hspace{4em}}c}
\begin{array}{c}
F_{2}(3;L,R)\\[-1pt]
\begin{array}{c|ccccc}
L\backslash R & 0&1&2&3&4\\\hline
0 & \cdot & 1 & 2 & \cdot & 1\\
1 & 3 & \cdot & 1 & \cdot & 3\\
2 & 3 & 2 & \cdot & \cdot & 2\\
3 & \cdot & \cdot & \cdot & \cdot & \cdot\\
4 & 1 & 3 & 1 & \cdot & \cdot
\end{array}
\end{array} &
\begin{array}{c}
F_{2}(4;L,R)\\[-1pt]
\begin{array}{c|ccccc}
L\backslash R & 0&1&2&3&4\\\hline
0 & \cdot & 1 & 3 & 1 & \cdot\\
1 & 2 & \cdot & 1 & 1 & \cdot\\
2 & 2 & 2 & \cdot & 1 & \cdot\\
3 & 2 & 3 & 2 & \cdot & \cdot\\
4 & \cdot & \cdot & \cdot & \cdot & \cdot
\end{array}
\end{array}
\end{array}}
\end{array}
\]
}

\paragraph{Position 3.}

{\small
\setlength{\arraycolsep}{2.5pt}
\setlength{\abovedisplayskip}{\abovedisplayshortskip}
\setlength{\belowdisplayskip}{4pt}
\setlength{\belowdisplayshortskip}{4pt}
\[
\begin{array}{c@{\hspace{4em}}c@{\hspace{4em}}c}
\begin{array}{c}
F_{3}(0;L,R)\\[-1pt]
\begin{array}{c|ccccc}
L\backslash R & 0&1&2&3&4\\\hline
0 & \cdot & \cdot & \cdot & \cdot & \cdot\\
1 & \cdot & \cdot & 1 & 2 & 2\\
2 & \cdot & 1 & \cdot & 3 & 2\\
3 & \cdot & 2 & 3 & \cdot & 3\\
4 & \cdot & 1 & 3 & 3 & \cdot
\end{array}
\end{array} &
\begin{array}{c}
F_{3}(1;L,R)\\[-1pt]
\begin{array}{c|ccccc}
L\backslash R & 0&1&2&3&4\\\hline
0 & \cdot & \cdot & 2 & 3 & 3\\
1 & \cdot & \cdot & \cdot & \cdot & \cdot\\
2 & 1 & \cdot & \cdot & 1 & 2\\
3 & 2 & \cdot & 2 & \cdot & 3\\
4 & 3 & \cdot & 3 & 1 & \cdot
\end{array}
\end{array} &
\begin{array}{c}
F_{3}(2;L,R)\\[-1pt]
\begin{array}{c|ccccc}
L\backslash R & 0&1&2&3&4\\\hline
0 & \cdot & 3 & \cdot & 2 & 1\\
1 & 1 & \cdot & \cdot & 1 & 2\\
2 & \cdot & \cdot & \cdot & \cdot & \cdot\\
3 & 2 & 2 & \cdot & \cdot & 1\\
4 & 2 & 2 & \cdot & 3 & \cdot
\end{array}
\end{array}\\
\noalign{\vskip 0.6em}
\multicolumn{3}{c}{%
\begin{array}{c@{\hspace{4em}}c}
\begin{array}{c}
F_{3}(3;L,R)\\[-1pt]
\begin{array}{c|ccccc}
L\backslash R & 0&1&2&3&4\\\hline
0 & \cdot & 2 & 1 & \cdot & 3\\
1 & 1 & \cdot & 2 & \cdot & 3\\
2 & 2 & 3 & \cdot & \cdot & 2\\
3 & \cdot & \cdot & \cdot & \cdot & \cdot\\
4 & 1 & 1 & 3 & \cdot & \cdot
\end{array}
\end{array} &
\begin{array}{c}
F_{3}(4;L,R)\\[-1pt]
\begin{array}{c|ccccc}
L\backslash R & 0&1&2&3&4\\\hline
0 & \cdot & 3 & 3 & 2 & \cdot\\
1 & 3 & \cdot & 3 & 2 & \cdot\\
2 & 3 & 3 & \cdot & 2 & \cdot\\
3 & 2 & 1 & 3 & \cdot & \cdot\\
4 & \cdot & \cdot & \cdot & \cdot & \cdot
\end{array}
\end{array}
\end{array}}
\end{array}
\]
}

\paragraph{Position 4.}

{\small
\setlength{\arraycolsep}{2.5pt}
\setlength{\abovedisplayskip}{\abovedisplayshortskip}
\setlength{\belowdisplayskip}{4pt}
\setlength{\belowdisplayshortskip}{4pt}
\[
\begin{array}{c@{\hspace{4em}}c@{\hspace{4em}}c}
\begin{array}{c}
F_{4}(0;L,R)\\[-1pt]
\begin{array}{c|ccccc}
L\backslash R & 0&1&2&3&4\\\hline
0 & \cdot & \cdot & \cdot & \cdot & \cdot\\
1 & \cdot & \cdot & 2 & 2 & 1\\
2 & \cdot & 1 & \cdot & 1 & 3\\
3 & \cdot & 2 & 1 & \cdot & 1\\
4 & \cdot & 3 & 1 & 2 & \cdot
\end{array}
\end{array} &
\begin{array}{c}
F_{4}(1;L,R)\\[-1pt]
\begin{array}{c|ccccc}
L\backslash R & 0&1&2&3&4\\\hline
0 & \cdot & \cdot & 2 & 2 & 2\\
1 & \cdot & \cdot & \cdot & \cdot & \cdot\\
2 & 3 & \cdot & \cdot & 2 & 3\\
3 & 1 & \cdot & 3 & \cdot & 3\\
4 & 1 & \cdot & 3 & 3 & \cdot
\end{array}
\end{array} &
\begin{array}{c}
F_{4}(2;L,R)\\[-1pt]
\begin{array}{c|ccccc}
L\backslash R & 0&1&2&3&4\\\hline
0 & \cdot & 2 & \cdot & 3 & 2\\
1 & 1 & \cdot & \cdot & 3 & 2\\
2 & \cdot & \cdot & \cdot & \cdot & \cdot\\
3 & 1 & 2 & \cdot & \cdot & 1\\
4 & 3 & 3 & \cdot & 2 & \cdot
\end{array}
\end{array}\\
\noalign{\vskip 0.6em}
\multicolumn{3}{c}{%
\begin{array}{c@{\hspace{4em}}c}
\begin{array}{c}
F_{4}(3;L,R)\\[-1pt]
\begin{array}{c|ccccc}
L\backslash R & 0&1&2&3&4\\\hline
0 & \cdot & 3 & 2 & \cdot & 3\\
1 & 1 & \cdot & 3 & \cdot & 2\\
2 & 3 & 3 & \cdot & \cdot & 1\\
3 & \cdot & \cdot & \cdot & \cdot & \cdot\\
4 & 3 & 1 & 2 & \cdot & \cdot
\end{array}
\end{array} &
\begin{array}{c}
F_{4}(4;L,R)\\[-1pt]
\begin{array}{c|ccccc}
L\backslash R & 0&1&2&3&4\\\hline
0 & \cdot & 1 & 3 & 3 & \cdot\\
1 & 1 & \cdot & 3 & 2 & \cdot\\
2 & 1 & 1 & \cdot & 2 & \cdot\\
3 & 2 & 2 & 1 & \cdot & \cdot\\
4 & \cdot & \cdot & \cdot & \cdot & \cdot
\end{array}
\end{array}
\end{array}}
\end{array}
\]
}

\end{document}